\documentclass[10pt]{amsart}
\usepackage{amssymb}

\usepackage{graphicx}
\usepackage{xcolor}

\theoremstyle{plain}
\newtheorem{theorem}{Theorem}[section]
\newtheorem{corollary}[theorem]{Corollary}
\newtheorem{lemma}[theorem]{Lemma}

\newtheorem{proposition}[theorem]{Proposition}

\newtheorem{observation}[theorem]{Observation}

\theoremstyle{remark}

\newtheorem*{claim}{Claim}

\begin{document}

\title[Borel sets, Baire spaces, and filtrations]{$\kappa$-Borel sets, $\kappa$-Baire spaces, and\\ filtrations between topologies}

\author{S{\l}awomir Solecki}
\address{Department of Mathematics, Malott Hall, Cornell University, Ithaca, NY 14853}
\email{ssolecki@cornell.edu}

\begin{abstract} Filtrations are certain transfinite sequences of topologies increasing in strength and interpolating between two given topologies $\sigma$ and $\tau$, with $\tau$ being stronger than $\sigma$. We prove general results on stabilization at $\tau$ of filtrations interpolating between $\sigma$ and $\tau$. These topological results involve an interplay between $\kappa$-Borel sets with respect to the topology $\sigma$ and $\kappa$-Baireness of the topology $\tau$. 
\end{abstract}

\thanks{Research supported by NSF grant DMS-2246873.}

\keywords{Filtration between topologies, $\kappa$-Baire spaces, $\kappa$-Borel sets}

\subjclass[2010]{03E15, 54H05} 


\maketitle

\tableofcontents

\section{Introduction} 

{\bf All topologies are defined on a fixed set $X$.}

We continue our analysis, initiated in \cite{So3}, of increasing transfinite sequences of topologies interpolating between 
two given topologies $\sigma\subseteq \tau$. Such sequences of topologies are relevant to some descriptive set theoretic considerations; see 
\cite[Section 1]{Lo}, \cite[Sections 5.1--5.2]{BK}, \cite[Section 2]{Be}, 
\cite[Section 2]{So1}, \cite[Chapter 6]{Hj1}, \cite[Section 3]{FS}, \cite[Sections 2--4]{So2}, \cite{Hj2}, \cite{Dr}, and 
\cite[Sections 3--5]{BDNT}. 

To make the discussion precise, recall from \cite{So3} the notion of filtration. 
Let $\sigma\subseteq \tau$ be topologies, and let $\rho$ be an ordinal. A transfinite sequence $(\tau_\xi)_{\xi<\rho}$ 
of topologies is called a {\bf filtration from $\sigma$ to $\tau$} if
\begin{equation}\label{E:cot}
\sigma= \tau_0\subseteq \tau_1\subseteq \cdots \subseteq \tau_\xi \subseteq \cdots \subseteq \tau
\end{equation}
and, for each $\alpha<\rho$, if $F$ is $\tau_\xi$-closed for some $\xi<\alpha$, then 
\begin{equation}\label{E:intap2}
{\rm int}_{\tau_\alpha}(F) = {\rm int}_{\tau}(F). 
\end{equation}
Condition \eqref{E:intap2} asserts that $\tau_\alpha$ computes correctly, that is, in agreement with $\tau$, 
the interiors of sets that are simple from the point of view of $\alpha$, that is, sets that are $\tau_\xi$-closed with $\xi<\alpha$. 
A filtration from $\sigma$ to $\tau$ can be thought of as a walk from $\sigma$ to $\tau$ through intermediate topologies, where each step is required to contribute a nontrivial advance towards $\tau$, if such an advance is at all possible. 

There is always a trivial filtration from $\sigma$ to $\tau$---define $\tau_\xi$, for an ordinal $\xi$, by letting $\tau_0=\sigma$ and $\tau_\xi=\tau$ for $\xi>0$. Then, for each ordinal $\rho>0$, the sequence $(\tau_\xi)_{\xi<\rho}$ is a filtration from $\sigma$ to $\tau$. This is the fastest filtration from $\sigma$ to $\tau$, in the sense that, for each filtration $(\tau_\xi')_{\xi<\rho}$ from $\sigma$ to $\tau$, we have $\tau_\xi'\subseteq \tau_\xi$. More interestingly, another canonical filtration from $\sigma$ to $\tau$ was defined in \cite{So3}. Namely, a transfinite sequence of topologies, denoted by $(\sigma, \tau)_\xi$, for an ordinal $\xi$, was described in \cite[Section~2]{So3} with the property that, for each ordinal $\rho>0$, $((\sigma, \tau)_\xi)_{\xi<\rho}$ is a filtration from $\sigma$ to $\tau$, and, for an arbitrary filtration $(\tau_\xi')_{\xi<\rho}$ from $\sigma$ to $\tau$, one has $(\sigma, \tau)_\xi\subseteq \tau_\xi'$. So, 
$((\sigma, \tau)_\xi)_{\xi<\rho}$ is the slowest filtration from $\sigma$ to $\tau$. 
See Section~\ref{Su:slo} below for the definition of $(\sigma,\tau)_\xi$.

The following question about filtrations immediately presents itself. 
Given a filtration $(\tau_\xi)_{\xi<\rho}$ from $\sigma$ to $\tau$, does the filtration actually reach its goal $\tau$ and, if so, at what stage? 
That is, does $\tau_\xi=\tau$ for some $\xi<\rho$ and is there an upper estimate on such $\xi$? 
For example, does the slowest filtration $((\sigma, \tau)_\xi)_{\xi<\rho}$ reach $\tau$? 
The issue of this filtration reaching $\tau$ was left unresolved in \cite{So3}.

One expects a positive answer to the question above under an appropriate assumption---$\tau$ should be at a visible distance from $\sigma$. We phrase this as follows. Let $\kappa$ be an infinite cardinal such that $\tau$ is $\kappa$-Baire; see Section~\ref{Su:bai} for the definition of $\kappa$-Baire and for the existence of such $\kappa$ for each topology $\tau$. The assumption on the pair of topologies $\sigma\subseteq \tau$, which we will call the {\em appropriate assumption} in the discussion below, is then for $\tau$ to have a neighborhood basis that consists of sets that are $\kappa$-Borel with respect to $\sigma$; see Section~\ref{Su:pisi} for the definition of $\kappa$-Borel sets. In this situation, the distance from $\sigma$ to $\tau$ is quantified by the complexity of sets in the neighborhood basis of $\tau$ as measured by the classical descriptive set theoretic hierarchy of $\kappa$-Borel sets; see Section~\ref{Su:pisi} for the definition of this hierarchy.

Paper \cite{So3} gave some results that contained a positive answer to the question. The results supposed that the {\em appropriate assumption} holds with $\kappa=\omega_1$, supposed additionally that the topology $\tau$ is regular, and placed quite stringent assumptions of metrizability on all the topologies in the filtration and of Bairness on some of them. In particular, these results were not applicable to the slowest filtration $((\sigma, \tau_\xi))_{\xi<\rho}$.

In the present paper, we prove theorems that strengthen the results in \cite{So3}. Theorem~\ref{T:stab2} asserts that, under a mild structural assumption of semiregularity on $\tau$, if the pair of topologies $\sigma\subseteq \tau$ satisfies the {\em appropriate assumption} with some infinite cardinal $\kappa$, then each filtration from $\sigma$ to $\tau$ reaches $\tau$ at a step commensurate with the descriptive complexity of sets in a neighborhood basis of $\tau$. In this theorem, the descriptive complexity is measured with the ${\mathbf \Pi}$ side of the classical hierarchy of $\kappa$-Borel sets. 
This result is a direct strengthening of the main results of \cite{So3}. First, no assumptions are made on the topologies in the filtration and only the semiregularity condition is put on $\tau$. 
The second improvement consists of freeing $\kappa$ to be an arbitrary infinite cardinal (from $\kappa=\omega_1$) in the assumptions of 
$\kappa$-Baireness of $\tau$ and $\kappa$-Borelness with respect to $\sigma$ of sets in a basis of $\tau$. Because of its generality, Theorem~\ref{T:stab2} can be applied to the slowest filtration $((\sigma, \tau)_\xi)_{\xi<\rho}$ showing that it terminates at $\tau$ under the {\em appropriate assumption}. This is done in Corollary~\ref{C:stab2}. 
Additionally, we prove a second theorem, Theorem~\ref{T:stab3}, that has a similar to Theorem~\ref{T:stab2}, but weaker, conclusion under weaker assumptions on the complexity of a basis of $\tau$. In this theorem, the descriptive complexity is measured with the ${\mathbf \Sigma}$ side of the hierarchy of $\kappa$-Borel sets. This theorem implies a corresponding result on the slowest filtration; see Corollary~ \ref{C:stab3}. 

The two improvements over \cite{So3}---the lack of restrictions on the topologies in the filtration 
and the lack of restrictions on the infinite cardinal $\kappa$--- put the results on termination at $\tau$ in what appears to be their optimal form. They make new applications possible, for example, applications to the slowest filtration $((\sigma, \tau)_\xi)_{\xi<\rho}$ and to topologies $\tau$ that do not satisfy any Baireness assumptions (as each topology is $\omega$-Baire). Further, the generalization to an arbitrary infinite cardinal $\kappa$ adds interest to the interaction between the degree of Bairness of $\tau$ with the degree of Borelness of a basis of $\tau$.

\section{Background information} 

We recall here some topological notions and establish notation and our conventions related to $\kappa$-Baireness and $\kappa$-Borelness. We also provide proofs of some relevant basic observations.

{\bf We fix a topology $\tau$.} 

{\bf We fix an infinite cardinal $\kappa$.}

\subsection{Basic notions and notation}

We write 
\[
{\rm cl}_\tau\;\hbox{ and }\;{\rm int}_\tau
\]
for the operations of closure and interior with respect to $\tau$. 

If $x\in X$, by a {\bf neighborhood of $x$} we understand a subset of $X$ that contains $x$ in its $\tau$-interior. 
A {\bf neighborhood basis of $\tau$} is a family $\mathcal A$ of subsets of $X$ such that for each $x\in X$ and each open set $B$ containg  $x$, there exists $A\in {\mathcal A}$ such that $x$ is in the $\tau$-interior of $A$ and $A\subseteq B$. A {\bf neighborhood $\pi$-basis of $\tau$} is a family $\mathcal A$ of subsets of $X$ such that for each nonempty $\tau$-open set $B$ there 
exists $A\in {\mathcal A}$ such that $A$ has nonempty interior and $A\subseteq B$.  So a neighborhood basis or a neighborhood $\pi$-basis  need not consist of open sets.

Given a family of topologies $S$, we write 
\[
\bigvee S
\]
for the topology whose 
basis consist of sets of the form $U_0\cap \cdots \cap U_n$, where each $U_i$, $i\leq n$, is $\sigma$-open for some $\sigma\in S$. This is the smallest 
topology containing each topology in $S$.

\subsection{$\kappa$-Baireness}\label{Su:bai} 

A set is called {\bf $\kappa$-meager} if it is the union of $<\kappa$ many $\tau$-nowhere dense sets. 
Observe that $\omega$-meager sets are $\tau$-nowhere dense sets. 
The topology $\tau$ is called {\bf $\kappa$-Baire} if the complement of each $\kappa$-meager set is dense.

\begin{observation}\label{O:bai} 
Either $\tau$ is $\kappa$-Baire for each infinite cardinal $\kappa$, or there is a largest infinite cardinal $\kappa$ with the property that $\tau$ is $\kappa$-Baire. 
\end{observation}

\begin{proof} Assume there is an infinite cardinal $\lambda$ such that $\tau$ is not $\lambda$-Baire. Note that each topology is $\omega$-Baire. 
Consider the set 
\[
\{ \kappa\mid \hbox{ $\kappa$ is a cardinal and $\tau$ is $\kappa$-Baire}\}.
\]
It is non-empty, as $\omega$ is in it, and bounded by $\lambda$. Its supremum is a cardinal $\kappa_0$. Directly from the definition we see that $\tau$ is $\kappa_0$-Baire.  
\end{proof}

\subsection{$\kappa$-Borel sets, ${\mathbf \Pi}$ and ${\mathbf \Sigma}$ classes of the Borel hierarchy}\label{Su:pisi} 

We define $\kappa$-Borel sets and their stratification into $\mathbf \Pi$ and $\mathbf \Sigma$ classes. 
Some care in formulating these definitions is needed as we want to incorporate the case of singular cardinals $\kappa$ in the right way. 
In the case of regular $\kappa$ the definitions we give coincides with the naive definitions---see Observations~\ref{O:bore}, \ref{O:borep}, and \ref{O:bores} below. 

First, for a cardinal $\nu$, let
\[
{\rm bor}(\nu)
\] 
be the smallest family of sets containing all closed sets and closed under taking complements and unions, and therefore also intersections, of families of cardinality $\leq \nu$. Define {\bf $\kappa$-Borel sets} as 
\[
\bigcup \{ {\rm bor}(\nu) \mid \nu\hbox{ a cardinal number},\, \nu<\kappa\}. 
\]
Observe that $\omega$-Borel sets form the algebra of sets generated by closed or open sets. 
The case $\kappa=\omega_1$ is the case studied in classical descriptive set theory; see \cite{Ke}.  

\begin{observation}\label{O:bore}
Assume $\kappa$ is regular. The family of $\kappa$-Borel sets is equal to the smallest family containing all closed sets and closed under taking complements and unions, and therefore also intersections, of families of cardinality $<\kappa$.
\end{observation}

\begin{proof} The inclusion $\subseteq$ is immediate and does not use regularity of $\kappa$. 
If $\kappa$ is regular, then clearly 
\[
\bigcup \{ {\rm bor}(\nu) \mid \nu\hbox{ a cardinal number},\, \nu<\kappa\}. 
\]
is closed under taking complements and unions of $<\kappa$ sets, which shows the inclusion $\supseteq$. 
\end{proof}

To clarify further our definition of $\kappa$-Borel sets, note that if $\kappa$ is singular, then the family of sets defined by the condition in Observation~\ref{O:bore} is equal to $\kappa^+$-Borel sets and not $\kappa$-Borel sets. 

We describe now a well known stratification of the family of $\kappa$-Borel sets.  
We define 
classes ${\mathbf \Pi}^{\kappa,0}_{1+\xi}$ and ${\mathbf \Sigma}^{\kappa,0}_{1+\xi}$, for $\xi<\kappa$. To incorporate singular cardinals $\kappa$, the same type of care is needed here 
as in the definition of $\kappa$-Borel sets. 
When $\kappa=\omega_1$, the classes above form the well-studied hierarchy of Borel sets; see \cite[Section~11]{Ke}. We conform to the tradition of enumerating these classes starting with $1$ rather than $0$.

For a cardinal $\nu>0$ and an ordinal $\xi<\nu^+$, we define families 
\[
{\mathbf P}^\nu_\xi.
\] 
Let 
${\mathbf P}^\nu_0$ 
be the family of all closed sets and, for $0<\xi<\nu^+$, let 
${\mathbf P}^\nu_\xi$
consist of intersections of subfamilies of cardinality $\leq \nu$ of complements of sets in $\bigcup_{\gamma<\xi}{\mathbf P}^\nu_\gamma$.

For an ordinal $\xi<\kappa$, define
\[
{\mathbf \Pi}^{\kappa,0}_{1+\xi} = \bigcup \{ {\mathbf P}^\nu_\xi\mid \nu \hbox{ a cardinal},\, \nu<\kappa,\, \xi<\nu^+\}
\]

\begin{observation}\label{O:borep}
\begin{enumerate}
\item[(i)] ${\mathbf \Pi}^{\kappa,0}_1$ is the class of all closed sets. 

\item[(ii)] If $\kappa$ is regular, then, 
for $0<\xi<\kappa$, ${\mathbf \Pi}^{\kappa,0}_{1+\xi}$ consists of all intersections of $<\kappa$ many complements of sets in 
$\bigcup_{\gamma<\xi} {\mathbf \Pi}^{\kappa,0}_{1+\gamma}$. 
\end{enumerate} 
\end{observation}

\begin{proof} (i) is obvious as ${\mathbf \Pi}^{\kappa,0}_{1}$ and each family ${\mathbf P}^{\nu}_{0}$, for cardinals $\nu<\kappa$, consist of exactly the closed sets. 

The inclusion $\subseteq$ in (ii) is easy to see and does not use the assumption that $\kappa$ is regular. Take $A$ in ${\mathbf \Pi}^{\kappa,0}_{1+\xi}$. Then $A$ is in 
${\mathbf P}^\nu_\xi$ for some $\nu<\kappa$ and $\xi<\nu^+$. So $A$ is the intersection of $\leq\nu$ complements of sets in 
$\bigcup_{\gamma<\xi}{\mathbf P}^\nu_\gamma$, and therefore it is the intersections of $<\kappa$ many complements of sets in 
$\bigcup_{\gamma<\xi} {\mathbf \Pi}^{\kappa,0}_{1+\gamma}$.

To prove the inclusion $\supseteq$ in (ii), fix $A$, for whichthere are a cardinal $\mu<\kappa$ and sets $B_\eta$ for $\eta<\mu$ such that  
\[
A= \bigcap_{\eta<\mu} (X\setminus B_{\eta})
\]
with $B_\eta$ in ${\mathbf \Pi}^{\kappa,0}_{1+\gamma_\eta}$ for some $\gamma_\eta<\xi$, for all $\eta<\mu$. By definition of 
${\mathbf \Pi}^{\kappa,0}_{1+\gamma_\eta}$, 
there exist cardinals $\mu_\eta<\kappa$, for $\eta<\mu$, such that $\gamma_\eta<\mu_\eta^+$ and $B_\eta$ is in 
${\mathbf P}^{\mu_\eta}_{\gamma_\eta}$. Since $\kappa$ is regular and $\mu<\kappa$, there exists a cardinal $\nu<\kappa$ such that $\mu_\eta\leq\nu$ for 
all $\eta<\mu$. It follows from this inequality that $\gamma_\eta<\nu^+$ and $B_\eta$ is in ${\mathbf P}^{\nu}_{\gamma_\eta}$
for all $\eta<\mu$. We can increase $\nu$ so that $\xi<\nu^+$. It follows that $A$ is in  ${\mathbf P}^{\nu}_{\xi}$, so in ${\mathbf \Pi}^{\kappa,0}_{1+\xi}$ as required. 
\end{proof}

Similarly, we define, for a cardinal $\nu>0$ and an ordinal $\xi<\nu^+$, 
\[
{\mathbf S}^\nu_\xi.
\] 
Let 
${\mathbf S}^\nu_0$ 
be the family of all open sets and, for $0<\xi<\nu^+$, let 
${\mathbf S}^\nu_\xi$
consist of unions of subfamilies of cardinality $\leq \nu$ of complements of sets in $\bigcup_{\gamma<\xi}{\mathbf S}^\nu_\gamma$. 

We make the following observation that is easy to prove by induction on $\xi$. 

\begin{observation}\label{O:comp} Let $\nu<\kappa$ be a cardinal and let $\xi<\nu^+$ be an ordinal. 
\begin{enumerate} 
\item[(i)]  ${\mathbf S}^\nu_\xi$ consists of complements of sets in ${\mathbf P}^\nu_\xi$. 

\item[(ii)] ${\mathbf S}^{\nu}_{\xi} \subseteq {\mathbf P}^{\nu}_{\xi+1}$ and ${\mathbf P}^{\nu}_{\xi} \subseteq {\mathbf S}^{\nu}_{\xi+1}$. 
\end{enumerate}
\end{observation}

For an ordinal $\xi<\kappa$, define
\[
{\mathbf \Sigma}^{\kappa,0}_{1+\xi} = \bigcup \{ {\mathbf S}^\nu_\xi\mid \nu \hbox{ a cardinal},\, \nu<\kappa,\, \xi<\nu^+\}
\]
The proof of the following observation is analogous to the proof of Observation~\ref{O:borep} and we omit it.

\begin{observation}\label{O:bores} 
\begin{enumerate} 
\item[(i)] ${\mathbf \Sigma}^{\kappa,0}_1$ is the class of all open sets. 

\item[(ii)] If $\kappa$ is regular, then, 
for $0<\xi<\kappa$, ${\mathbf \Sigma}^{\kappa,0}_{1+\xi}$ consists of all unions of $<\kappa$ many complements of sets in 
$\bigcup_{\gamma<\xi} {\mathbf \Sigma}^{\kappa,0}_{1+\gamma}$. 
\end{enumerate} 
\end{observation}

Below, we have the expected result on the relationship between $\kappa$-Borel sets and the $\mathbf \Pi$ and $\mathbf \Sigma$ classes. 
It is an immediate consequence of Observation~\ref{O:comp}. 

\begin{observation}\label{O:alg} 
We have 
\begin{equation}\notag
\bigcup_{\xi<\kappa} {\mathbf \Sigma}^{\kappa,0}_{1+\xi} = \bigcup_{\xi<\kappa} {\mathbf \Pi}^{\kappa,0}_{1+\xi}, 
\end{equation}
and the family above is the family of all $\kappa$-Borel sets. 
\end{observation}

Finally, we connect $\kappa$-Borelness with $\kappa$-meagerness. 
A set $A$ is said to have the {\bf $\kappa$-Baire property} if $A= (U\setminus M_1)\cup M_2$, where $U$ is open and $M_1$ and $M_2$ are $\kappa$-meager. 

\begin{observation}\label{O:bpr} Each $\kappa$-Borel set has the $\kappa$-Baire property with respect to each topology $\tau'$ with $\tau\subseteq \tau'$. 
\end{observation} 

\begin{proof} If $A$ is $\kappa$-Borel, then, by Observation~\ref{O:alg}, $A$ is in ${\mathbf S}^\nu_\xi$ for some cardinal $\nu<\kappa$ and some ordinal $\xi<\nu^+$. 
Fix these $\nu$ and $\xi$. Consider the family of sets of the form $(U\setminus M_1)\cup M_2$, where $U$ is $\tau'$-open and $M_1,\, M_2$ are unions of $\leq\nu$ $\tau'$-nowhere dense sets. Of course, each such set has the $\kappa$-Baire property with respect to $\tau'$. So, it will suffice to show that $A$ is of this form. If $\xi=0$, this is clear since all $\tau$-open sets are $\tau'$-open. To prove it for $0<\xi<\nu^+$, it suffices to show that sets of the above form are closed under complements and unions of families of cardinality $\leq\nu$. This is done by the standard argument using the facts that for each $\tau'$-open set $U$, ${\rm cl}_{\tau'}(U)\setminus U$ is $\tau'$-nowhere dense and that a union of $\leq\nu$ sets that are unions of $\leq\nu$ nowhere dense sets is a union of $\leq\nu$ nowhere dense sets. Observe only that proving this last assertion splits into two cases--- $\nu<\omega$ and $\omega\leq \nu$. 
\end{proof}

\subsection{Semiregularity and $\pi$-semiregularity} 
We recall the following notion due to Stone \cite[Chapter III, Definition 19]{St}.
A topology is called {\bf semiregular} if it is generated by its regular open sets, that is, by sets that are interiors of their closures. See \cite{En} for more information on this notion.
To compare semiregularity with regularity of a topology, note that a topology is regular if and only if, for each open set $U$, there exists a family $\mathcal F$ of closed sets such that 
\[
U= \bigcup \{ F\mid F\in {\mathcal F}\} = \bigcup \{ {\rm int}(F)\mid F\in {\mathcal F}\}. 
\]
If we recall that a set is regular open precisely when it is equal to the interior of a closed set, it becomes clear that a topology is semiregular if and only if, for each open set $U$, there exists a family $\mathcal F$ of closed sets such that 
\[
U= \bigcup\{ {\rm int}(F)\mid F\in {\mathcal F}\}. 
\]
In particular, it follows that each regular topological space is semiregular. 

A topology is called {\bf $\pi$-semiregular} if each nonempty open set contains a non\-empty regular open set.

\section{Stabilization of filtrations and descriptive complexity}\label{S:stdes}

\subsection{Theorems on stabilization of filtrations}

Here is our first main result. 
We write $(\tau_\xi)_{\xi\leq\rho}$ for $(\tau_\xi)_{\xi<\rho+1}$.

\begin{theorem}\label{T:stab2}
Let $\sigma\subseteq \tau$ be topologies with $\tau$ semiregular. Let $\kappa$ be an infinite cardinal such that $\tau$ is $\kappa$-Baire, and 
let $1\leq \alpha\leq \kappa$ be an ordinal. 
Assume that $\tau$ has a neighborhood basis consisting of sets in $\bigcup_{\xi<\alpha}{\mathbf \Pi}^{\kappa,0}_{1+\xi}$ with respect to $\sigma$. 
\begin{enumerate}
\item[(i)] If $\alpha$ is a successor and $(\tau_\xi)_{\xi\leq\alpha}$ is a filtration from $\sigma$ to $\tau$, then $\tau= \tau_\alpha$. 

\item[(ii)] If $\alpha$ is limit and $(\tau_\xi)_{\xi<\alpha}$ is a filtration from $\sigma$ to $\tau$, then $\tau=\bigvee_{\xi<\alpha}\tau_\xi$. 
\end{enumerate} 
\end{theorem}

By Observation~\ref{O:bai}, there always exists an infinite cardinal $\kappa$ as in the assumptions of Theorem~\ref{T:stab2} above. In fact, by this observation, either each infinite cardinal works, or there is a largest infinite cardinal that works. 

Theorem~\ref{T:stab2} gives a common generalization and strengthening of the main results in \cite{So3}, namely \cite[Theorem~4.1 and Corollary~4.9]{So3}. Aside from replacing $\omega_1$ with an arbitrary infinite cardinal $\kappa$, it removes the assumptions of 
metrizability of the topologies $\tau_\xi$, $\xi<\alpha$, and of 
Baireness of the topology $\tau_\alpha$ in the filtration $(\tau_\xi)_{\xi\leq\alpha}$, and it weakens the assumption of regularity of $\tau$ to semiregularity.

We register one immediate corollary that is obtained by specializing Theorem~\ref{T:stab2} above to $\kappa=\alpha=\omega$ and $\kappa=\alpha=\omega_1$. We emphasize these cases as they are the most interesting ones from the point of view of descriptive set theory.

\begin{corollary}
Let $\sigma\subseteq \tau$ be topologies with $\tau$ semiregular. 
\begin{enumerate} 
\item[(i)] Assume that $\tau$ has a neighborhood basis included in the algebra of sets generated by $\sigma$. 
If $(\tau_\xi)_{\xi<\omega}$ is a filtration from $\sigma$ to $\tau$, then $\tau=\bigvee_{\xi<\omega}\tau_\xi$. 

\item[(ii)] Assume that $\tau$ is $\omega_1$-Baire and has a neighborhood basis included in the $\sigma$-algebra of sets generated by $\sigma$. 
If $(\tau_\xi)_{\xi<\omega_1}$ is a filtration from $\sigma$ to $\tau$, then $\tau=\bigvee_{\xi<\omega_1}\tau_\xi$. 
\end{enumerate} 
\end{corollary}

We also have a theorem for the ${\mathbf \Sigma}^{\kappa,0}_{1+\xi}$ classes.

\begin{theorem}\label{T:stab3}
Let $\sigma\subseteq \tau$ be topologies, 
with $\tau$ $\pi$-semiregular. Let $\kappa$ be an infinite cardinal such that $\tau$ is 
 $\kappa$-Baire, and let $1\leq \alpha\leq \kappa$ be an ordinal. Assume that $\tau$ has a neighborhood $\pi$-basis consisting of sets in $\bigcup_{\xi<\alpha}{\mathbf \Sigma}^{\kappa,0}_{1+\xi}$ with respect to $\sigma$. 

If $(\tau_\xi)_{\xi<\alpha}$ is a filtration from $\sigma$ to $\tau$, then 
each nonempty $\tau$-open set contains a nonempty $\big(\bigvee_{\xi<\alpha} \tau_\xi\big)$-open set. 
\end{theorem}

\subsection{The slowest filtration}\label{Su:slo}

Let $\sigma$ and $\tau$ be two topologies with $\sigma\subseteq\tau$. We consider here a canonical filtration from $\sigma$ to $\tau$ that was introduced in \cite[Section~2]{So3}. The transfinite sequence of topologies $(\sigma, \tau)_\xi$ is defined by recursion on $\xi$ for all ordinals $\xi$. 
Let 
\[
(\sigma, \tau)_0=\sigma. 
\]
Given an ordinal $\xi>0$, 
if $(\sigma, \tau)_\gamma$ are defined for all $\gamma<\xi$, then $(\sigma, \tau)_\xi$ is the family of all unions of sets of the form 
\[
U\cap {\rm int}_\tau(F)
\]
where, for some $\gamma<\xi$,  $U$ is $(\sigma, \tau)_\gamma$-open and $F$ is $(\sigma, \tau)_\gamma$-closed. Observe that
\begin{equation}\label{E:vest}
\bigvee_{\xi<\alpha}(\sigma,\tau)_\xi = (\sigma, \tau)_\alpha, \;\hbox{ if }\alpha \hbox{ is a limit ordinal}. 
\end{equation}

Proposition~\ref{P:slo}, which is \cite[Proposition~2.2]{So3}, justifies regarding $((\sigma,\tau)_\xi)_{\xi<\rho}$ as the slowest filtration from $\sigma$ to $\tau$. 

\begin{proposition}[\cite{So3}]\label{P:slo} 
Let $\sigma\subseteq \tau$ be topologies. 
\begin{enumerate}
\item[(i)] $((\sigma,\tau)_\xi)_{\xi<\rho}$ is a filtration from $\sigma$ to $\tau$, for each ordinal $\rho$. 

\item[(ii)] If $(\tau_\xi)_{\xi<\rho}$ is a filtration from $\sigma$ to $\tau$, then $(\sigma,\tau)_\xi\subseteq \tau_\xi$, for each $\xi<\rho$. 
\end{enumerate}
\end{proposition}

The following corollary is an immediate consequence of Proposition~\ref{P:slo}(i), equation \eqref{E:vest}, and Theorem~\ref{T:stab2}.

\begin{corollary}\label{C:stab2}
Let $\sigma\subseteq \tau$ be topologies with $\tau$ semiregular. Let $\kappa$ be an infinite cardinal such that $\tau$ is $\kappa$-Baire, and 
let $1\leq \alpha\leq \kappa$ be an ordinal. 
Assume that $\tau$ has a neighborhood basis consisting of sets in $\bigcup_{\xi<\alpha}{\mathbf \Pi}^{\kappa,0}_{1+\xi}$ with respect to $\sigma$. Then $\tau= (\sigma,\tau)_\alpha$. 
\end{corollary}

In a manner analogous to Corollary~\ref{C:stab2}, Corollary~\ref{C:stab3} follows from Proposition~\ref{P:slo}(i), \eqref{E:vest}, and Theorem~\ref{T:stab3}.

\begin{corollary}\label{C:stab3}
Let $\sigma\subseteq \tau$ be topologies with $\tau$ $\pi$-semiregular. 
Let $\kappa$ be an infinite cardinal such that $\tau$ is $\kappa$-Baire, and 
let $1\leq \alpha\leq \kappa$ be an ordinal. Assume that $\tau$ has a neighborhood $\pi$-basis consisting of sets in $\bigcup_{\xi<\alpha}{\mathbf \Sigma}^{\kappa,0}_{1+\xi}$ with respect to $\sigma$. 
Then each nonempty $\tau$-open set contains a nonempty
$(\sigma, \tau)_\alpha$-open set. 

In fact, if $\alpha$ is a successor, then each nonempty $\tau$-open set contains a nonempty
$(\sigma, \tau)_{\alpha-1}$-open set. 
\end{corollary}

Note that by Proposition~\ref{P:slo}(ii) the two corollaries are, in fact, equivalent to the theorems from which they are derived.

\section{Proofs of the main results}

The proofs below build on the past experience from \cite{So2} and \cite{So3}. 
In comparison with \cite{So3}, the main new contributions are the notions of $\xi_-$-slight and $\xi_+$-slight sets and the arguments in Section~\ref{Su:endg}. The introduction of $\xi_-$-slight and $\xi_+$-slight sets allows us to state an asymmetric version of Lemma~\ref{L:slal} (with sets $A\setminus F$ and $F\setminus A$ being small in different ways) and to run the arguments in Section~\ref{Su:aux} in much greater generality than in \cite{So3}.

{\bf We fix topologies $\sigma\subseteq \tau$. 

We fix an infinite cardinal $\kappa$ such that $\tau$ is $\kappa$-Baire.

The phrases concerning meagerness, non-meagerness, and comeagerness refer to the corresponding notions with respect to the topology $\tau$.}

\subsection{Reformulations of Theorems~\ref{T:stab2} and \ref{T:stab3}} 
In order to state our arguments in a uniform manner, it will be convenient to introduce a piece of notation and a convention. 
For $1\leq \alpha\leq\kappa$, let 
\[
\alpha\oplus 1 = 
\begin{cases}
\alpha+1,& \text{ if $\alpha$ is a successor ordinal;}\\
\alpha, & \text{ if $\alpha$ is a limit ordinal.}
\end{cases}
\]
Observe that all the filtrations from $\sigma$ to $\tau$ in Theorems~\ref{T:stab2} and \ref{T:stab3} are of the form $(\tau_\xi)_{\xi<\rho}$ for some $\rho\leq\kappa$. Each filtration of this form can be extended to a filtration $(\tau_\xi)_{\xi<\kappa}$ from $\sigma$ to $\tau$ by letting 
$\tau_\xi=\tau$ for all $\xi$ with $\rho\leq \xi<\kappa$. Thus, without loss of generality and by convention, we will assume that all the filtrations are of the form $(\tau_\xi)_{\xi<\kappa}$.

Now, using the notation and convention set up above, the conclusions (i) and (ii) of Theorem~\ref{T:stab2} can be amalgamated into one statement so the theorem reads:
 
\medskip

{\em Let $1\leq \alpha\leq \kappa$. 
Assume that $\tau$ is semiregular and has a neighborhood basis consisting of sets in $\bigcup_{\xi<\alpha}{\mathbf \Pi}^{\kappa,0}_{1+\xi}$ with respect to $\sigma$. 

If $(\tau_\xi)_{\xi<\kappa}$ is a filtration from $\sigma$ to $\tau$, then $\tau=\bigvee_{\xi<\alpha\oplus 1}\tau_\xi$.}

\medskip

Similarly, Theorem~\ref{T:stab3} can be phrased as follows: 

 \medskip

{\em Let $1\leq \alpha\leq \kappa$. Assume that $\tau$ is $\pi$-semiregular and has a neighborhood $\pi$-basis consisting of sets in $\bigcup_{\xi<\alpha}{\mathbf \Sigma}^{\kappa,0}_{1+\xi}$ with respect to $\sigma$. 

If $(\tau_\xi)_{\xi<\kappa}$ is a filtration from $\sigma$ to $\tau$, then 
each nonempty $\tau$-open set contains a nonempty $\big(\bigvee_{\xi<\alpha} \tau_\xi\big)$-open set.}

\subsection{Lemmas}\label{Su:aux}

{\bf We fix a sequence $(\tau_\xi)_{\xi<\kappa}$ of topologies fulfilling 
condition \eqref{E:cot}.}  

{\bf We fix a cardinal $\nu<\kappa$.} 

As usual, by $\nu^+$, we denote the cardinal successor of $\nu$. Obviously, we have $\nu^+\leq\kappa$.

With $\nu$ fixed, we define the following notion that will be crucial in the proof. We define families of subsets of $X$ called $\xi_-$-slight and $\xi_+$-slight. 
This is done by simultaneous recursion on $\xi< \nu^+$. 
The only $0_-${\bf -slight} set is the empty set. 
For $0<\xi<\nu^+$, let {\bf $\xi_-$-slight} sets be the smallest family of sets closed under taking subsets and unions of $\leq \nu$ sets and containing 
\begin{enumerate}
\item[---] $\gamma_+$-slight sets for all $\gamma<\xi$ and 

\item[---] $\tau_\xi$-closed sets with empty $\tau$-interiors. 
\end{enumerate} 
For $\xi<\nu^+$, let {\bf $\xi_+$-slight} sets be the family of sets that are $\tau_\xi$-locally $\xi_-$-slight, that is, $A\subseteq X$ is $\xi_+$-slight if there exists a covering of $A$ by $\tau_\xi$-open sets $U$ such that $A\cap U$ is $\xi_-$-slight.

\begin{lemma}\label{L:inclu} 
Let $\xi<\nu^+$. 
\begin{enumerate}
\item[(i)] If $\gamma<\xi$, then each set that is $\gamma_-$-slight or $\gamma_+$slight is both $\xi_-$-slight and $\xi_+$-slight. 

\item[(ii)] All $\xi_-$-slight and all $\xi_+$-slight sets are $\nu^+$-meager, in particular, they are $\kappa$-meager. 
\end{enumerate} 
\end{lemma} 

\begin{proof} (i) is obvious. 

(ii) 
The proof is by induction on $\xi$. The conclusion is obvious for $\xi=0$ since $\emptyset$ is the only $0_-$-slight set and it is also the only $0_+$-slight set. To see the conclusion for $\xi>0$, note that $\nu^+$-meager sets are closed under taking unions of $\leq \nu$ sets and subsets, and under the following operation:  if $A$ is covered by a family $\mathcal U$ of $\tau$-open sets $U$ such that $A\cap U$ is $\nu^+$-meager,
then $A$ is $\nu^+$-meager. This last assertion is proved by modifying the argument used to prove Banach's theorem that an arbitrary union of open meager sets is meager. Let ${\mathcal V}$ be a maximal 
family of pairwise disjoint $\tau$-open sets $V$ such that $V\subseteq \bigcup {\mathcal U}$ and $A\cap V$ is $\nu^+$-meager. By the assumption on $\mathcal U$ and by maximality of $\mathcal V$, we have that $\bigcup {\mathcal V}$ is $\tau$-dense in $\bigcup {\mathcal U}$. Since $\bigcup {\mathcal V}$ is $\tau$-open, it follows that $\bigcup {\mathcal U}\setminus \bigcup {\mathcal V}$ is $\tau$-nowhere dense, and therefore so is $A\setminus \bigcup {\mathcal V}$. So, it suffices to show that $A\cap \bigcup {\mathcal V}$ is $\nu^+$-meager. For each $V\in {\mathcal V}$, let $A^V_\eta$, $\eta<\nu$, be $\tau$-nowhere dense and such that 
\[
A\cap V= \bigcup_{\eta<\nu} A^V_\eta.
\]
Since $\mathcal V$ consists of pairwise disjoint $\tau$-open sets, we have that, for each $\eta<\nu$, the set $\bigcup_{V\in {\mathcal V}} A^V_\eta$ is $\tau$-nowhere dense. Thus, 
\[
A\cap \bigcup {\mathcal V} = \bigcup_{\eta<\nu} \big( \bigcup_{V\in {\mathcal V}} A^V_\eta\big)
\]
is $\nu^+$-meager, as required.

The above observations, in addition to our inductive assumption that $\gamma_+$-slight sets are $\nu^+$-meager and the observation that $\tau_\xi$-closed sets with empty $\tau$-interiors are $\tau$-nowhere dense (since $\tau_\xi$-closed sets are $\tau$-closed), we see that $\xi_-$-slight are $\nu^+$-meager and then that $\xi_+$-slight sets are $\nu^+$-meager. 
\end{proof}

To state the next two lemmas in their optimal form, it will be necessary to recall the notion of weak filtration from \cite{So3}. 
We say that $(\tau_\xi)_{\xi<\rho}$ is a {\bf weak filtration from $\sigma$ to $\tau$} provided \eqref{E:cot} holds and, for each $\alpha<\rho$, 
if $F$ is $\tau_\xi$-closed for some $\xi<\alpha$, then 
\begin{equation}\label{E:intap}
{\rm int}_{\tau_\alpha}(F) \hbox{ is $\tau$-dense in }  {\rm int}_{\tau}(F). 
\end{equation}
It is clear that each filtration is a weak filtration.

{\bf From this point on, we assume that the sequence $(\tau_\xi)_{\xi<\kappa^+}$ is a weak filtration.}
 
Recall the definition of ${\mathbf P}^\nu_\xi$ from Section~\ref{Su:pisi}.

\begin{lemma}\label{L:slal}
If $A$ is in ${\mathbf P}^\nu_\xi$,
for $\xi< \nu^+$, then there exists a $\tau_\xi$-closed set $F$ such that $A\setminus F$ is $\xi_+$-slight and $F\setminus A$ is $\xi_-$-slight. 
\end{lemma}

\begin{proof} For $A$ in ${\mathbf P}^\nu_\xi$, with $\xi< \nu^+$, we define 
\[
c_\xi(A) = X\setminus \bigcup\{ U\mid A\cap U\hbox{ is $\xi_-$-slight and $U$ is $\tau_\xi$-open}\}. 
\]
Obviously the set $c_\xi(A)$ is $\tau_\xi$-closed. 

To prove the lemma, it is sufficient to show that if $A$ is in ${\mathbf P}^\nu_\xi$, then 
\begin{enumerate}
\item[(a)] $A\setminus c_\xi(A)$ is $\xi_+$-slight; 

\item[(b)] $c_\xi(A)\setminus A$ is $\xi_-$-slight.
\end{enumerate}
Point (a) is clear from the definition of $\xi_+$-slight sets. We prove (b) by induction on $\xi$. 

If $\xi=0$, then $A$ is in ${\mathbf P}^\nu_0$, that is, $A$ is $\sigma$-closed 
and $c_0(A)=A$. So $c_0(A)\setminus A=\emptyset$, which is $0_-$-slight. 

Assume that $\xi>0$ and that (b) holds for all $\gamma<\xi$. Fix $A$ in ${\mathbf P}^\nu_\xi$.
There exists a sequence 
$B_\eta$, $\eta<\nu$, with $B_\eta$ in ${\mathbf P}^\nu_{\gamma_\eta}$,
for some $\gamma_\eta<\xi$, 
such that 
\[
X\setminus A = \bigcup_{\eta<\nu} B_\eta. 
\]
We have 
\[
c_\xi(A)\setminus A = c_\xi(A)\cap \bigcup_{\eta<\nu} B_\eta \subseteq \bigcup_{\eta<\nu} \bigl(c_\xi(A) \cap c_{\gamma_\eta}(B_\eta)\bigr) \cup \bigl( B_\eta\setminus c_{\gamma_\eta}(B_\eta)\bigr).  
\] 
Since $\xi_-$-slight sets are closed under unions of $\nu$ sets, it suffices to show that, for each $\eta<\nu$,
\begin{enumerate}
\item[($\alpha$)] $B_\eta\setminus c_{\gamma_\eta}(B_\eta)$ is $\xi_-$-slight and 

\item[($\beta$)] $c_\xi(A)\cap c_{\gamma_\eta}(B_\eta)$ is $\xi_-$-slight. 
\end{enumerate}

Point ($\alpha$) is clear since, by our inductive assumption, 
$B_\eta\setminus c_{\gamma_\eta}(B_\eta)$ 
is $(\gamma_\eta)_+$-slight and, therefore, by Lemma~\ref{L:inclu}(i), $\xi_-$-slight as $\gamma_\eta<\xi$. 

We prove ($\beta$). Since the set $c_\xi(A)\cap c_{\gamma_\eta}(B_\eta)$ is $\tau_\xi$-closed, by definition of $\xi_-$-slight sets, it suffices to show that 
\begin{equation}\label{E:abe}
{\rm int}_\tau\big(c_\xi(A)\cap c_{\gamma_\eta}(B_\eta)\big)=\emptyset. 
\end{equation} 

We make two observations. First, let $V$ be a $\tau_\xi$-open set such that 
\begin{equation}\notag
V\cap c_\xi(A)\not=\emptyset. 
\end{equation} 
Then, by the definition of $c_\xi(A)$, we see that $A\cap V$ is not $\tau_\xi$-locally $\xi_-$-slight. Thus, by the definition of $\xi_+$-sets, we get that, for each $\tau_\xi$-open set $V$,
\begin{equation}\label{E:vla}
V\cap c_\xi(A)\not=\emptyset \Rightarrow A\cap V \hbox{ is not $\xi_+$-slight}. 
\end{equation} 
Second, we note that 
\[
A\cap c_{\gamma_\eta}(B_\eta)\subseteq c_{\gamma_\eta}(B_\eta)\setminus B_\eta, 
\]
so, by our inductive assumption, $A\cap c_{\gamma_\eta}(B_\eta)$ is $(\gamma_\eta)_-$-slight. Since $\gamma_\eta<\xi$, it follows by Lemma~\ref{L:inclu}(i) that 
\begin{equation}\label{E:asm}
A\cap c_{\gamma_\eta}(B_\eta)\hbox{ is $\xi_+$-slight}. 
\end{equation} 

Now, to prove \eqref{E:abe}, assume towards a contradiction that there exists a nonempty $\tau$-open set $U$ such that 
\[
U\subseteq c_\xi(A)\;\hbox{ and }\; U\subseteq c_{\gamma_\eta}(B_\eta). 
\]
Since $c_{\gamma_\eta}(B_\eta)$ is $\tau_{\gamma_\eta}$-closed and $\gamma_\eta<\xi$, by definition of weak filtration, there exists a $\tau_\xi$-open set $V$ such that 
\[
V\subseteq c_{\gamma_\eta}(B_\eta)\;\hbox{ and }\; V\cap U\not=\emptyset. 
\]
Then, by \eqref{E:asm}, $A\cap V$ is $\xi_+$-slight. 
By \eqref{E:vla}, this implies that 
$V\cap c_\xi(A)=\emptyset$ contradicting $U\subseteq c_\xi(A)$ and $V\cap U\not=\emptyset$. 
\end{proof}

\begin{lemma}\label{L:stab}
Let $A$ be in ${\mathbf P}^\nu_\xi$,
for some $\xi<\nu^+$, and let $B\subseteq A$ be such that $B\cap U$ is not $\xi_+$-slight, for each $\tau_\xi$-open set $U$ with $B\cap U\not=\emptyset$. Then ${\rm cl}_{\tau_\xi}(B)\setminus A$ is $\xi_-$-slight. 
\end{lemma}

\begin{proof}
Let $A$ and $B\subseteq A$ be as in the assumptions. By Lemma~\ref{L:slal}, there exists a $\tau_\xi$-closed 
set $F$ such that $A\setminus F$ is $\xi_+$-slight and $F\setminus A$ is $\xi_-$-slight. The assumption on $B$ gives that $B\setminus F$ is empty, so $B\subseteq F$. 
Since $F$ is $\tau_\xi$-closed, we have ${\rm cl}_{\tau_\xi}(B)\subseteq F$, which gives 
\[
{\rm cl}_{\tau_\xi}(B)\setminus A\subseteq F\setminus A. 
\] 
It follows that ${\rm cl}_{\tau_\xi}(B)\setminus A$ is $\xi_-$-slight. 
\end{proof}

{\bf From this point on, we assume that $(\tau_\xi)_{\xi<\kappa}$ is a filtration from $\sigma$ to $\tau$}.

\begin{lemma}\label{L:last}
Let $\gamma<\beta <\nu^+$. Let $A$ be in ${\mathbf P}^\nu_\gamma$,
and let $B$ be $\tau$-open and such that 
$B\setminus A$ is $\kappa$-meager. Then there exists a $\tau_\beta$-open set $V$ such that 
\begin{equation}\label{E:xvm} 
B\subseteq V\,\hbox{ and }\,V\setminus A\hbox{ is $\kappa$-meager}. 
\end{equation} 
\end{lemma}

\begin{proof} Let 
\[
F={\rm cl}_{\tau_\gamma}(A\cap B)\hbox{ and }\, V={\rm int}_{\tau_\beta}(F). 
\]
Clearly $V$ is $\tau_\beta$-open. We show that \eqref{E:xvm} holds for it.

By our assumption $B\setminus A$ is $\kappa$-meager, so for each $\tau_\gamma$-open set $U$, 
$(B\cap U)\setminus A$ is $\kappa$-meager. 
So, by Lemma~\ref{L:inclu}(ii) and by our assumption that $\tau$ is $\kappa$-Baire, we get that, 
for each $\tau_\gamma$ open set $U$, if $B\cap U\not=\emptyset$, then $A\cap B\cap U$ is not $\gamma_+$-slight. Thus, since $A\cap B\subseteq A$, by Lemma~\ref{L:stab}, $F\setminus A$ is $\gamma_-$-slight. Now, by Lemma~\ref{L:inclu}(ii), we get 
\begin{equation}\notag
F\setminus A\hbox{ is $\kappa$-meager.}
\end{equation} 
Since 
\[
V \setminus A\subseteq F\setminus A, 
\]
we see that $V\setminus A$ is $\kappa$-meager and the second part of \eqref{E:xvm} follows. .

Observe that since $\tau$ is $\kappa$-Baire and $B$ is $\tau$-open, $A\cap B$ is $\tau$-dense in $B$, so $B\subseteq F$ as $\tau_\gamma\subseteq \tau$. Since $\gamma<\beta$ and $(\tau_\xi)_{\xi<\kappa}$ is a filtration, and since $B$ is $\tau$-open, we have 
\begin{equation}\label{E:atop}
V= {\rm int}_{\tau_\beta}(F) = {\rm int}_{\tau}(F)\supseteq B, 
\end{equation} 
and \eqref{E:xvm} follows. 
\end{proof}

\subsection{Proofs of Theorems~\ref{T:stab2} and \ref{T:stab3}}\label{Su:endg} 

We start with Theorem~\ref{T:stab2}.

\begin{proof}[Proof of Theorem~\ref{T:stab2}]
Set 
\[
\tau^\alpha = \bigvee_{\xi<\alpha\oplus 1} \tau_\xi.
\]

We show that each $\tau$-regular open set is $\tau^\alpha$-open. Proving this statement suffices since $\tau$ is assumed to be semiregular. 

\begin{claim}
If $A\subseteq X$ is a $\tau$-neighborhood of $x$, then $A$ is $\kappa$-comeager in a $\tau^\alpha$-neighborhood of $x$. 
\end{claim} 

\noindent{\em Proof of Claim.} We can suppose that $A$ is in ${\mathbf P}^\nu_\xi$, for some cardinal $\nu<\kappa$ and 
some ordinal $\xi$ with $\xi<\alpha$ and $\xi<\kappa^+$. 
Let 
\[
B= {\rm int}_\tau(A).
\]
Note that $B\subseteq A$ is $\tau$-open. We use Lemma~\ref{L:last} to get a $\tau^\alpha$-open set $V$ with \eqref{E:xvm}. 
This can be accomplished since if $\alpha$ is a successor, then $\tau^\alpha=\tau_\alpha$ and $\xi<\alpha$, and we can apply Lemma~\ref{L:last} with $\beta=\alpha$ and $\gamma=\xi$. If $\alpha$ is limit, let $\beta$ be such that $\xi<\beta<\alpha$, and we 
apply Lemma~\ref{L:last} with this $\beta$ and $\gamma=\xi$. 
Since $x\in B$, \eqref{E:xvm} implies  
\begin{equation}\notag
x\in V\;\hbox{ and }\;V\setminus A\hbox{ is $\kappa$-meager}, 
\end{equation} 
and $V$ is the $\tau^\alpha$-neighborhood of $x$ desired by the conclusion of the claim. 

\medskip

Now, we show that 
\begin{equation}\label{E:mcl}
U\subseteq {\rm int}_{\tau^\alpha}\big( {\rm cl}_\tau(U)\big),\hbox{ for each $\tau$-open set $U$.}
\end{equation}
If this is not the case, then there exists $x\in U$ such that $x\not\in {\rm int}_{\tau^\alpha}\big( {\rm cl}_\tau(U)\big)$, that is, 
\begin{equation}\label{E:xin}
x\in {\rm cl}_{\tau^\alpha}\big( X\setminus {\rm cl}_\tau(U)\big).
\end{equation} 
Since $x\in U$, the claim implies that there exists a $\tau^\alpha$-open set $V$ such that $x\in V$ and $V\setminus U$ is $\kappa$-meager. Now, by \eqref{E:xin}, we have 
\[
V\cap \big( X\setminus {\rm cl}_\tau(U)\big)\not= \emptyset.
\]
But this set is $\tau$-open and is included in $V\setminus U$, which, by $\tau$ being $\kappa$-Baire, contradicts $V\setminus U$ being $\kappa$-meager. This contradiction proves \eqref{E:mcl}.  

It follows from \eqref{E:mcl} and from $\tau^\alpha\subseteq \tau$ that for each $\tau$-regular open set $U$, we have 
\[
 {\rm int}_{\tau^\alpha}\big( {\rm cl}_\tau(U)\big)\subseteq {\rm int}_\tau\big( {\rm cl}_\tau(U)\big) = U\subseteq {\rm int}_{\tau^\alpha}\big( {\rm cl}_\tau(U)\big). 
\]
In particular, we have 
\[
U= {\rm int}_{\tau^\alpha}\big( {\rm cl}_\tau(U)\big). 
\]
So, all $\tau$-regular open sets are $\tau^\alpha$-open.
 \end{proof}

We now prove the second main result of the paper.

\begin{proof}[Proof of Theorem~\ref{T:stab3}] 
Fix $0<\alpha\leq \kappa$, and set 
\[
\tau^\alpha = \bigvee_{\xi<\alpha} \tau_\xi.
\]

\begin{claim}
If $A$ is ${\mathbf \Sigma}^{\kappa,0}_{1+\xi}$, for some $\xi<\alpha$, and is $\kappa$-nonmeager, then $A$ is $\kappa$-comeager in a nonempty $\tau^\alpha$-open set. 
\end{claim} 

\noindent{\em Proof of Claim.} If $\xi=0$, then $A$ is $\sigma$-open and so $\tau^\alpha$-open. Assume $0< \xi<\alpha$. 
We can find a cardinal $\nu<\kappa$ such that $\xi<\nu^+$ and $A\in {\mathbf S}^\nu_\xi$. 
Then, by Observation~\ref{O:comp}(i), there are sets $B_\eta$, for $\eta<\nu$, such that $B_\eta$ is in ${\mathbf P}^{\nu}_{\gamma_\eta}$, for some $\gamma_\eta<\xi$, and 
\[
A=\bigcup_{\eta< \nu} B_\eta.
\]
By Observation~\ref{O:bpr}, each $B_\eta$ has the $\kappa$-Baire property with respect to $\tau$. Therefore, since $\tau$ is $\kappa$-Baire and since $A$ is $\kappa$-nonmeager, there exists a $\tau$-open set $U$ and $\eta_0<\nu$ such that 
\[
U\not= \emptyset \;\hbox{ and }\;U\setminus B_{\eta_0}\hbox{ is $\kappa$-meager}. 
\]
By Lemma~\ref{L:last}, we get a $\tau^\alpha$-open set $V$ such that 
\[
U\subseteq V\;\hbox{ and }\; V\setminus B_{\eta_0}\hbox{ is $\kappa$-meager}. 
\]
As in the previous proof, to reach the above consequence of Lemma~\ref{L:last}, we consider separately the case of a successor $\alpha$ (when $\tau^\alpha=\tau_{\alpha-1}$) and the case of a limit $\alpha$. It follows that $V$ is a nonempty $\tau^\alpha$-open set, in which $A$ is comeager, and the claim is proved. 
\medskip

We show that, for each $\tau$-open set $U$, 
\begin{equation}\label{E:mcl2}
U\setminus {\rm int}_{\tau^\alpha}\big( {\rm cl}_\tau(U)\big)\hbox{ is $\tau$-nowhere dense}.
\end{equation}
If, for some $\tau$-open set $U$, this is not the case, then there exists a nonempty $\tau$-open set $V$ such that 
\begin{equation}\label{E:mid}
U\setminus {\rm int}_{\tau^\alpha}\big( {\rm cl}_\tau(U)\big)\hbox{ is $\tau$-dense in $V$}.
\end{equation} 
Since ${\rm int}_{\tau^\alpha}\big( {\rm cl}_\tau(U)\big)$ is $\tau$-open, as it is $\tau^\alpha$-open, we see that \eqref{E:mid} implies 
\begin{equation}\label{E:not}
V\cap {\rm int}_{\tau^\alpha}\big( {\rm cl}_\tau(U)\big)=\emptyset.
\end{equation} 
On the other hand, we have that $V\cap U\not=\emptyset$. By our assumptions on $\tau$, $V\cap U$ contains a $\kappa$-nonmeager subset in 
${\mathbf \Sigma}^{\kappa,0}_\xi$ for some $\xi<\alpha$. So by the claim, there exists a nonempty $\tau^\alpha$-open set $W$ such that $U\cap V$ is $\kappa$-comeager in $W$. It follows that $W\setminus U$ and $W\setminus V$ are $\kappa$-meager. So, since $\tau$ is $\kappa$-Baire, we have 
\[
W\subseteq {\rm cl}_\tau(U)\;\hbox{ and }\; W\subseteq {\rm cl}_\tau(V). 
\]
From the first inclusion above, using $\tau^\alpha$-openness of $W$, we get 
\[
W\subseteq {\rm int}_{\tau^\alpha}\big( {\rm cl}_\tau(U)\big).
\]
From the second inclusion, using \eqref{E:not} and $\tau$-openness of ${\rm int}_{\tau^\alpha}\big( {\rm cl}_\tau(U)\big)$, we get
\[
W\cap {\rm int}_{\tau^\alpha}\big( {\rm cl}_\tau(U)\big)=\emptyset. 
\]
So $W=\emptyset$. This contradiction proves \eqref{E:mcl2}.  

It follows from $\tau^\alpha\subseteq \tau$ that for each $\tau$-regular open set $U$, we have 
\[
 {\rm int}_{\tau^\alpha}\big( {\rm cl}_\tau(U)\big)\subseteq {\rm int}_\tau\big( {\rm cl}_\tau(U)\big) = U. 
\]
On the other hand, by \eqref{E:mcl2} and by $\tau$ being $\kappa$-Baire, we have that ${\rm int}_{\tau^\alpha}\big( {\rm cl}_\tau(U)\big)$ 
is nonempty if $U$ is nonempty. 
In particular, ${\rm int}_{\tau^\alpha}\big( {\rm cl}_\tau(U)\big)$ is a nonempty $\tau^\alpha$-open set included in the nonempty $\tau$-regular open set  $U$. 
The conclusion of the theorem follows from $\tau$ being $\pi$-semiregular. 
\end{proof}

\bigskip 

\noindent {\bf Acknowledgement.} I would like to thank Alan Dow for pointing out to me reference \cite{St} as the origin of the notion of semiregularity. I would also like to thank Mateusz Lichman for our conversations about \cite{So3} that brought my attention back to the subject.

\end{document}